\documentclass[12pt]{amsart}

\usepackage{amsfonts,amsmath,amssymb,amsthm}
\usepackage[latin1]{inputenc}
\usepackage{tikz}
\usetikzlibrary{calc}
\usepackage[all]{xy}
\usepackage{float}

\newcommand{\Cc}{\mathbb{C}} 
\newcommand{\Pp}{\mathbb{P}}
\newcommand{\Rr}{\mathbb{R}}

\renewcommand {\le}{\leqslant}
\renewcommand {\ge}{\geqslant}

\newcommand{\grad}{\mathop{\mathrm{grad}}\nolimits} 
\newcommand{\Baff}{{\mathcal{B}_\mathit{\!aff}}}
\newcommand{\Binf}{{\mathcal{B}_\mathit{\!inf}}}
\newcommand{\B}{{\mathcal{B}}}
\newcommand{\gen}{\mathrm{gen}}

\newcommand{\defi}[1]{\emph{#1}}

{\theoremstyle{plain}
\newtheorem{theorem}{Theorem}    % theorem with number
\newtheorem*{theorem*}{Theorem}
\newtheorem{lemma}[theorem]{Lemma}       % lemma with number
\newtheorem{proposition}[theorem]{Proposition}      % lemma with number
\newtheorem{proposition*}{Proposition} 
\newtheorem*{theoremA*}{Theorem A}
\newtheorem*{theoremAA*}{Theorem A'}
\newtheorem*{theoremB*}{Theorem B}

}
{\theoremstyle{remark}

\newtheorem*{remark*}{Remark}  % theorem without number
\newtheorem*{question*}{Question}
}

\title{Topology of generic line arrangements}

\date{\today}

\author{Arnaud Bodin}
\email{Arnaud.Bodin@math.univ-lille1.fr}

\address{Laboratoire Paul Painlev\'e, Math\'ematiques, Universit\'e 
Lille 1, 59655 Villeneuve d'Ascq Cedex, France}



\begin{document}

\begin{abstract}
Our aim is to generalize the result that two generic complex line arrangements are equivalent.
In fact for a line arrangement $\mathcal{A}$ we associate its defining polynomial $f= \prod_i (a_ix+b_iy+c_i)$, 
so that $\mathcal{A}= (f=0)$. We prove that the defining polynomials of two generic 
line arrangements are, up to a small deformation, topologically equivalent.
In higher dimension the related result is that within a family of equivalent hyperplane arrangements
the defining polynomials are topologically equivalent.
\end{abstract}

\maketitle

%%%%%%%%%%%%%%%%%%%%%%%%%%%%%%%%%%%%%%%%%%%%%%%%%%%%%%%%%%%%%%%%%
%%%%%%%%%%%%%%%%%%%%%%%%%%%%%%%%%%%%%%%%%%%%%%%%%%%%%%%%%%%%%%%%%

\section*{Introduction}

We start with a basic result that says that generic complex line arrangements are topologically unique.
\defi{Generic} means that there is no triple point. For a generic line arrangements its \defi{combinatorics}
is $\{p_1,\ldots,p_\ell\}$ where $p_j$ denotes the number of parallel lines in a given direction indexed by $j$.

\begin{theoremA*}
\label{th:A}
Any two generic complex line arrangements having the same combinatorics are topologically equivalent.
\end{theoremA*}
This result is known to be false for real line arrangements, even in the case where there is no parallel lines.

\bigskip

One of the key-point is to construct a family of generic line arrangements
that links the two arrangements. If we already start from a family $\mathcal{A}_t$ of hyperplane arrangements 
(or even of subspace arrangements) then the constancy of the combinatorics structure
implies the topological equivalence by a result of R.~Randell, \cite{Ra}. 
More precisely two arrangements $\{H_1,\ldots,H_d\}$ and $\{H_1',\ldots,H_d'\}$ have the \defi{same lattice}
if for any $I\subset \{1,\ldots,d\}$, $\dim \bigcap_{i\in I} H_i = \dim \bigcap_{i\in I} H_i'$.

\begin{theorem*}[Randell]
Let $\mathcal{A}_t$ be a smooth family of hyperplane arrangements. 
If the lattice of $\mathcal{A}_t$ remains the same for all $t\in[0,1]$ then
the arrangements $\mathcal{A}_0$ and $\mathcal{A}_1$ are topologically equivalent.
\end{theorem*}

\bigskip

Our goal is to study not only the arrangement but also the global behaviour of the defining polynomial 
of the arrangement.
To an arrangement composed with lines of equation $a_ix+b_iy+c_i=0$
we associated its \defi{defining polynomial} $f(x,y)= \prod_{i=1}^d
(a_ix+b_iy+c_i)$.
We are interested in the arrangement $\mathcal{A}=f^{-1}(0)$
but also in the other curve levels $f^{-1}(c)$. In particular some algebraic curves $f^{-1}(c)$ may have singularities.
We say that $f$ is \defi{Morse outside the arrangement}
if the other singularities are ordinary double points, with distinct critical values.
\begin{theoremAA*}
\label{th:AA}
Let $\mathcal{A}_0$ and $\mathcal{A}_1$ be two generic line arrangements with the same combinatorics and $f_0$ and $f_1$ their 
defining polynomials.
\begin{itemize}
  \item Up to a small deformation of $\mathcal{A}_0$ and $\mathcal{A}_1$ the polynomials $f_0$ and $f_1$ are topologically equivalent.
  \item If $f_0$ and $f_1$ are Morse outside the arrangement, then the polynomials $f_0$ and $f_1$ are topologically equivalent.
\end{itemize}
\end{theoremAA*}

One step of the proof is to connect $f_0$ and $f_1$ by a family $(f_t)$ of 
polynomials defining generic arrangements. Then one can apply a global version
of L\^e-Ramanujam theorem to prove the topological equivalence.

The distinction between the topological equivalence of arrangements and of polynomials is important.
If two polynomials are topologically equivalent then the arrangements are topologically equivalent,
but also any fibre $f_0^{-1}(c)$ is equivalent to some fibre $f_1^{-1}(c')$.

The reciprocal is false: two arrangements can be equivalent but not their defining polynomials.
Let $f_t = xy(x+y-4)(x-ty).$
For non-zero $t,t'$ the arrangements of $f_t$ and $f_{t'}$ are
equivalent. Let $j, \bar j$ the roots of $z^2+z+1$. Then for
$t,t' \notin \{-1,0,j,\bar j\}$, $f_t$ is Morse outside the arrangement: $f$ has
two critical values (outside $0$), that correspond to double
points. For $t=j$ or $t=\bar j$, $f_t$ has only one critical
value outside $0$ which is $3(t-1)$, the corresponding fibre
$f_t^{-1}(3(t-1))$ has two double points. Then for $t=j$ (or
$t=\bar j$) and for $t' \notin \{-1,0,j,\bar j\}$ the
polynomials $f_t$ and $f_{t'}$ are not topologically equivalent
(but they have equivalent arrangements).

\bigskip

What happens in higher dimension for hyperplane arrangements?
And if the arrangements are non-generic?
We are able to prove the following:
\begin{theoremB*}
\label{th:B}
If $(f_t)_{t\in[0,1]}$ is a smooth 
family of equivalent complex hyperplane arrangements in $\Cc^n$ (with $n\neq 3$) 
that are Morse outside the arrangement, then the polynomials $f_0$ and $f_1$ 
are topologically equivalent.
\end{theoremB*}

This result is more in the spirit of Randell's theorem since we start from a family.
Note also that in the hypothesis we can substitute ``equivalent arrangements'' by ``the same lattice''.

\medskip

\emph{Acknowledgements:} I thank Michael Falk for discussions and the reference to Randell's work.

%%%%%%%%%%%%%%%%%%%%%%%%%%%%%%%%%%%%%%%%%%%%%%%%%%%%%%%%%%%%%%%%%%%%%%%
%%%%%%%%%%%%%%%%%%%%%%%%%%%%%%%%%%%%%%%%%%%%%%%%%%%%%%%%%%%%%%%%%%%%%%%

\section{Definitions}

An \defi{arrangement} in $\Cc^n$  is a finite collection
of hyperplanes $\mathcal{A}=\{ H_i\}$.

Two arrangements $\mathcal{A}$ and $\mathcal{A'}$ are \defi{topologically equivalent} 
if the pair $(\Cc^n,\mathcal{A})$ is homeomorphic to the pair $(\Cc^n,\mathcal{A}')$, in other words 
there exists a homeomorphism $\Phi : \Cc^n \to \Cc^n$ such that $\phi(\mathcal{A})= \mathcal{A'}$.
\bigskip

To a hyperplane $H_i$ we associate an affine equation $\ell_i(x)=0$ 
for some $\ell_i(x) =  a_{i,1} x_1 +\cdots + a_{i,n} x_n + c_i$.

The \defi{defining polynomial} of $\mathcal{A}=\{ H_i\}_{i=1,\ldots,d}$
is $f = \prod_{i=1}^d \ell_i$. 
Our point of view is not only to consider the arrangement $(f=0)$ but also all other hypersurfaces defined by equations
of type $(f=c)$.

We will say that the arrangement (or $f$) is \defi{Morse outside the arrangement} if 
\begin{itemize}
  \item its defining polynomial $f$ has only a finite number of critical points
in $\Cc^n \setminus \mathcal{A}$;
  \item this critical points are ordinary double points;
  \item the critical values are distinct.
\end{itemize}
This hypothesis does not concern the points of the arrangement $\mathcal{A}$ identified with $(f=0)$.

\bigskip

In this part we first focus on complex line arrangements, that is to say configurations of lines in $\Cc^2$.
A line arrangement is \defi{generic} if the following conditions hold:
\begin{itemize}
  \item there are no triple points,
  \item not all the lines are parallel.
\end{itemize}
The first condition means that three lines cannot have a common point of intersection.
To a generic line arrangement $\prod_{i=1}^d(a_ix+b_iy+c_i)$ we associate a set of integers.
Let the set of \defi{directions} be $\{ (a_i:b_i) \}_{i=1,\ldots, d} \subset \Pp^1$.
To each direction $\delta_j$ let $p_j$ be the number of lines parallel to $\delta_j$.
The \defi{combinatorics} of a generic line arrangement is $(p_1,\ldots,p_\ell)$. This list is unique up to permutation
and we have $\sum_{i=1}^\ell p_i = d$.

For example if no lines are parallel then $\ell=d$ and the combinatorics is $(1,1,\ldots,1)$.
A combinatorics equals to $(3,2,1,1)$ denotes a generic line arrangement of $7$ lines, three of them being parallel
(in one direction) another two being parallel (in another direction) and the two last ones not being parallel to any
other lines (see the picture below).
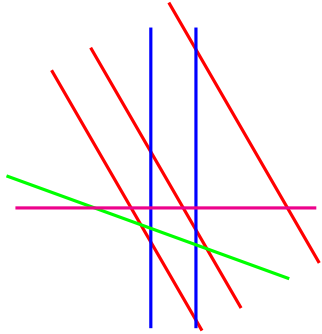
\begin{figure}[H]
\begin{tikzpicture}[scale=2]
\begin{scope}[red, rotate=30]
 \draw[very thick] (0,-1)--(0,1);
 \draw[very thick] (-0.3,-1)--(-0.3,1);  
 \draw[very thick] (0.6,-1)--(0.6,1);  
\end{scope}
\begin{scope}[blue, rotate=0]
 \draw[very thick] (-0.1,-1)--(-0.1,1);  
 \draw[very thick] (0.2,-1)--(0.2,1);  
\end{scope}
\begin{scope}[green, rotate=70]
 \draw[very thick] (-0.35,-1)--(-0.35,1);  
\end{scope}
\begin{scope}[magenta, rotate=-90] 
 \draw[very thick] (0.2,-1)--(0.2,1);  
\end{scope}
\end{tikzpicture}  
\end{figure}

%%%%%%%%%%%%%%%%%%%%%%%%%%%%%%%%%%%%%%%%%%%%%%%%%%%%%%%%%%%%%%%%%%%%%%%
\section{Generic complex line arrangements are unique}

\begin{theorem}
\label{th:generic}
Let $\mathcal{A}_0$ and $\mathcal{A}_1$ be two generic complex line arrangements with the same combinatorics.
Then the arrangements are equivalent, that is to say the pairs $(\Cc^2,\mathcal{A}_0)$ and $(\Cc^2,\mathcal{A}_1)$ 
are homeomorphic.

Suppose moreover that their defining polynomials $f_0$ and $f_1$ are Morse outside the arrangement.
Then the polynomials $f_0$ and $f_1$ are topologically equivalent.
\end{theorem}

\begin{remark*}
Theorem \ref{th:generic} is false in the realm of real line arrangements.
For example in the case of $6$ lines in the real planes there exist four different generic configurations 
that are not equivalent.
See \cite[appendix]{Fi}.
\end{remark*}

\begin{question*}
It would be interesting, given any arrangement (in particular for a generic one), to prove that it is possible to 
find an equivalent arrangement such that its defining polynomial is a Morse function outside the arrangement.  
\end{question*}

We will recall the definition of topological equivalence and other facts about the 
topology of polynomials in the next paragraph.

%%%%%%%%%%%%%%%%%%%%%%%%%%%%%%%%%%%%%%%%%%%%%%%%%%%%%%%%%%%%%%%%%%%%%%%
\section{Short review on topology of polynomials}

Let $f$ be  a polynomial in $\Cc[x_1,\ldots,x_n]$.
There exists a finite \defi{bifurcation set} $\B \subset \Cc$ such that
$$f : f^{-1}(\Cc\setminus \B) \longrightarrow \Cc \setminus \B$$
is a locally trivial fibration. In particular for $c_\gen \notin \B$,
$f^{-1}(c_\gen)$ is a \defi{generic fibre}.
If a fibre $f^{-1}(c)$ is singular then it is not generic, we denote
by $\Baff$ the set of all affine critical values.

For a value $c \in \Cc$, if there exists $0< \delta \ll 1$ and $R \gg 1$ such that
$f : f^{-1}\big(D^2_\delta(c)\big) \setminus B^{2n}_R(0) \longrightarrow D^2_\delta(c)$
is a locally trivial fibration we say that $c$ is a \defi{regular value at infinity}.
The set of \defi{irregular values at infinity} is denoted by $\Binf$.
We have that 
$$\B = \Baff \cup \Binf.$$

\bigskip

Two polynomials $f,g \in \Cc[x_1,\ldots,x_n]$ are \defi{topologically equivalent}
if there exist homeomorphisms $\Phi$ and $\Psi$ with a commutative diagram:
$$ \xymatrix{
{}\Cc^n \ar[r]^-\Phi\ar[d]_-{f}  & \Cc^n \ar[d]^-{g} \\
  \Cc \ar[r]_-\Psi        & \Cc
}
$$

In particular, if $\Psi(0)=0$, it implies that the pair $(\Cc^n, f^{-1}(0))$ is homeomorphic to
$(\Cc^n, g^{-1}(0))$.

\bigskip

A simple criterion to construct topologically equivalent polynomials is
to consider a family $(f_t)$ with some numerical invariants not depending on the parameter $t$.
We then apply the following global L\^e-Ramanujam theorem (\cite{Bo}, \cite{BT}):
\begin{theorem}
\label{th:mucst}
Let $n\neq3$.
If $\big(f_t(x_1,\ldots,x_n)\big)_{t\in[0,1]}$ is a continuous family of polynomials with isolated singularities such that
the degree $\deg f_t$, the number of critical values (affine and at infinity) $\#\B(t)$ 
and the Euler characteristic of a generic fibre $\chi(f^{-1}(c_\gen))$ remain constant
for all $t\in[0,1]$, then $f_0$ and $f_1$ are topologically equivalent.
\end{theorem}

This result will be our main tool to prove theorem \ref{th:generic} and theorem \ref{th:hyper}.

%%%%%%%%%%%%%%%%%%%%%%%%%%%%%%%%%%%%%%%%%%%%%%%%%%%%%%%%%%%%%%%%%%%%%%%
\section{Proof that generic line arrangements are unique}

\begin{lemma}
\label{lem:generic}
Let $\mathcal{A}_0$ and $\mathcal{A}_1$ be two generic complex line arrangements having the same combinatorics.
Then $\mathcal{A}_0$ and $\mathcal{A}_1$ can be linked by a continuous family of generic 
complex line arrangements $\big\{\mathcal{A}_t\big\}_{t\in[0,1]}$ having all the same combinatorics.
\end{lemma}

``Continuous family'' means that the equation of each line of $\mathcal{A}_t$, $a_i(t)x+b_i(t)y+c_i(t)$
has coefficients depending continuously on $t\in [0,1]$.

\begin{proof}
Let $\mathcal{A}$ be a generic line arrangement. Suppose that there are no horizontal line.
Choose a direction $d=(-b:a) \in \Pp^1$ and let $L_1,\ldots,L_k$ be the lines of $\mathcal{A}$ having direction $d$.
The equations of the $L_i$ are $a x+b y+c_i$. We will move these lines to an horizontal 
position while remaining in the set of generic arrangements.
We define $L_i(t)$ ($0\le t \le 1$) by the equation $a(t)x+by+c_i$ ($1\le i \le k$),
where $a(0)=a$ and $a(t) \to 0$ as $t\to 1$.
At some $t_0$ a line $L_i(t)$ may encounter one of the finite number of double points of the remaining set of lines 
$\mathcal{A}\setminus \{L_1,\ldots,L_k\}$. in such a case we redefine $L_i(t)$ for $t \in [t_0-\epsilon,t_0]$
by the equation $a(t)x+by+c_i(t)$ where $\epsilon>0$ is small enough and $c_i(t)$ is a small perturbation of $c_i$,
for instance $c_i(t)=c_i+ t -(t_0-\epsilon)$ in order to avoid the double point. We do a similar process if we 
encounter a direction of the remaining set of lines 
$\mathcal{A}\setminus \{L_1,\ldots,L_k\}$.
For $t\ge t_0$ we continue with the equation
$a(t)x+by+c_i(t_0)$. For $t=1$ all the lines $L_i(1)$ ($1\le i \le k$) are horizontal. 
And we can even move these horizontal lines to lines of equation
$y=i$ ($1\le i \le k$) (if no integer $1\le i \le k$ appears in the coordinates of the 
double points of the remaining lines).

We iterate the process to move another set of parallel lines to a given direction. Therefore any pair of 
generic arrangements can be linked by some continuous family of generic arrangements, 
the combinatorics remains the same.
\end{proof}

\begin{lemma}
\label{lem:nosinginf}
The defining polynomial of a complex line arrangement (not all lines being parallel) has no singularity at infinity.  
\end{lemma}

It holds in any dimension as we will prove latter in lemma \ref{lem:nosinginfbis}.
But in the two dimensional case we can provide simpler arguments.

\begin{proof}
Let $f(x,y) = \prod_{i=1}^d (a_ix+b_iy+c_i)$ be the defining polynomial of a line arrangement,
each direction corresponds to a point at infinity of $(f=0)$.
We homogenize $f(x,y)-c$ in $F(x,y,z)-cz^d$ and localize at a point at infinity. 
For instance if $\delta < d$ lines are parallel to the $y$-axis, they have intersection at infinity at $P=(0:1:0)$.
We may suppose that these lines have equations $x+c_i=0$, $i=1,\ldots,\delta$, the $c_i$ being pairwise distinct.
The localization of $f$ at $P$ is $f_P(x,z)=\prod_{i=1}^{\delta} (x+c_iz) - cz^d$.

It remains to prove that the topology of the germ $f_P$ is independent of $c$.
A first method is to see that its Newton polygon is independent of $c$ (because $\delta < d$)
and is Newton non-degenerate (because all the $c_i$ are distinct), then by Kouchnirenko's theorem 
\cite{Ko} the  local Milnor number $\mu_P(f_P)$ of the germ is constant.

Another method is to compute the resolution of $f$, by blow-ups of $f_P$ at $P$, and see that 
no critical value occurs (a good reference for different characterizations of irregular values at infinity is
\cite{Du}).
\end{proof}

\begin{lemma}
\label{lem:newcharac}
Let $f$ be the defining polynomial of a generic line arrangement having combinatorics $(p_1,\ldots,p_\ell)$.
Then the Euler characteristic of the generic fibre $\chi(f^{-1}(c_\gen))$ is
$$\chi(f^{-1}(c_\gen)) = -d(d-2) + \sum_{j=1}^\ell  p_j(p_j-1)$$
and 
$$\chi(f^{-1}(0)) = 1-\frac{(d-1)(d-2)}{2}+ \sum_{j=1}^\ell \frac{p_j(p_j-1)}{2}.$$
\end{lemma}

Notice that in proposition \ref{prop:genfiber} below we will be able to recover $\chi(f^{-1}(c_\gen))$ by proving
 that the generic fibre can be obtained by $d$ disks, and each pair of disks is connected by two bands.

\begin{proof}

As $(f=0)$ is a generic line arrangement, $(f=0)$ has the homotopy type of
$\frac{(d-1)(d-2)}{2}-\sum_{j=1}^\ell \frac{p_j(p_j-1)}{2}$ circles; it yields the formula for $\chi(f^{-1}(0))$.

Let $\mu(0)$ be the sum of Milnor numbers at singular points of $f^{-1}(0)$.
As each singularity in the arrangement is an ordinary double points, $\mu(0)$ equals the number of those double points.
Hence $\mu(0)= \frac{d(d-1)}{2} - \sum_{j=1}^\ell \frac{p_j(p_j-1)}{2}.$
Now, as there is no irregular value at infinity we have the formula :
$$\chi(f^{-1}(0)) -\chi(f^{-1}(c_\gen))=\mu(0)$$
that expresses that the number of vanishing cycles equals the number of singular points (counted with multiplicity)
yields the result for $\chi(f^{-1}(c_\gen))$.
\end{proof}

\begin{lemma}
\label{lem:cardsing}
For the defining polynomial $f$ of a hyperplane arrangement that is Morse outside 
the arrangement we have:
$$\# \B = 2 - \chi(f^{-1}(0)).$$
\end{lemma}

So that by lemma \ref{lem:newcharac} we have an explicit formula for $\#\B$.

\begin{proof}
By lemma \ref{lem:nosinginf} there is no irregular value at infinity.
The formula
$$1-\chi(f^{-1}(c_\gen)) = \sum_{c\in \B} \chi(f^{-1}(c)) -\chi(f^{-1}(c_\gen))$$
(that reformulates that the global Milnor number is the sum of the local Milnor numbers)
can be decomposed in
$$1-\chi(f^{-1}(c_\gen)) = \chi(f^{-1}(0)) -\chi(f^{-1}(c_\gen)) + 
\sum_{c\in \B\setminus\{0\}} \chi(f^{-1}(c)) -\chi(f^{-1}(c_\gen))$$
As $\B \setminus \{ 0 \}$ is composed only of affine critical values, that moreover are Morse critical values, 
we get
$\chi(f^{-1}(c)) -\chi(f^{-1}(c_\gen))=1$ for all $c \in \B\setminus\{0\}$;
that implies
$$\# \B = 2 - \chi(f^{-1}(0))$$
\end{proof}

\begin{proof}[Proof of theorem \ref{th:generic}]
The first step is to link the defining polynomials $f_0$ and $f_1$ by a family
$(f_t)_{t\in[0,1]}$ of polynomial such that the corresponding arrangement $\mathcal{A}_t$ are generic 
and have the same combinatorics, this is possible by lemma \ref{lem:generic}. 

To apply the global $\mu$-constant theorem (theorem \ref{th:mucst}) to our family $(f_t)$
we need also to choose $(f_t)$ such that $\deg f_t$, $\chi(f_t^{-1}(c_\gen))$ and $\#\B(t)$ are independent of $t$.
This is clear for the degree that equals the number of lines and in lemma \ref{lem:newcharac} 
we already proved that $\chi(f_t^{-1}(c_\gen))$ depends only on the combinatorics.

For $\#\B(t)$ the situation is more complicated. It is not always possible to find $(f_t)$ such that 
$\#\B(t)$ remains constant (see the example in the introduction).
However we will prove that in the set of generic arrangements of a given combinatorics there exists
a dense subset such that for these arrangements $\#\B$ is constant.

We fix a combinatorics $(p_1,\ldots,p_\ell)$ and consider the set $\mathcal{C}$ of all
polynomials defining a generic $d$-line arrangement having combinatorics equals to $(p_1,\ldots,p_\ell)$.
First of all $\mathcal{C}$ is a connected set (see lemma \ref{lem:generic}) and is a constructible set 
of the set of polynomials defining a $d$-line arrangement.
Now $\mathcal{C}$ is stratified by a finite number of constructible subsets $\mathcal{C}_i$
such that for each  $f \in \mathcal{C}_i$, $\# \B(f) = i$: 
there are only a finite number of critical values,  by lemma \ref{lem:nosinginf} there is no irregular value at infinity
and affine critical values are given by the vanishing of a resultant.
Hence one of this constructible subset, say $\mathcal{C}_{i_0}$, his dense (and contains a non-empty Zariski open set) 
in $\mathcal{C}$.

\medskip

We then can prove theorems A, A' and \ref{th:generic}.

There exist a small deformation $\tilde f_0$ of $f_0$  and a small deformation $\tilde f_1$ of $f_1$ and a continuous family $\tilde f_t$, $t\in [0,1]$ 
such that the arrangements defined by $f_0$ and $\tilde f_0$ (resp.~ $f_1$ and $\tilde f_1$) are equivalent
and $\tilde f_0$ and $\tilde f_1$ belongs to $\mathcal{C}_{i_0}$.

Now link $\tilde f_0$ to $\tilde f_1$ by a family $(\tilde f_t)$ of generic arrangements
having the same combinatorics and belonging to $\mathcal{C}_{i_0}$.
For $(\tilde f_t)$ we have the constancy of $\deg \tilde f_t$, $\chi(\tilde f_t^{-1}(c_\gen))$ and $\#\B(t)$
so that  we can now apply theorem \ref{th:mucst}, to conclude that $\tilde f_0$ and $\tilde f_1$ are topologically equivalent.
In particular the arrangements $\tilde f_0^{-1}(0)$ is topologically equivalent to $\tilde f_1^{-1}(0)$
so that the arrangement $f_0^{-1}(0)$ is topologically equivalent to $f_1^{-1}(0)$.

On the other hand if we start with $f_0$ (and $f_1$) Morse outside that arrangement then we 
already know that the corresponding constructible set $\mathcal{C}_{i_0}$ is the set of Morse functions outside the arrangement (because being Morse is an open condition)
and is non-empty (by hypothesis). Then we can directly build a family $(f_t)$ of Morse function outside the arrangement
 linking $f_0$ to $f_1$.
It is then clear that $\deg f_t$, $\chi(f_t^{-1}(c_\gen))$ and $\#\B(t)$ are constant (see lemma \ref{lem:cardsing} for instance)
and theorem \ref{th:mucst} implies that $f_0$ and $f_1$ are topologically equivalent.
\end{proof}

\begin{question*}
Prove that generic arrangements that are Morse outside the arrangement are dense in the set of generic arrangements.
By algebraic arguments, it is equivalent to find \emph{one} generic arrangement that is Morse outside the arrangement!
For instance for a complex line arrangement whose lines have real equations,  Varchenko \cite{Va} proved that
there is exactly one critical point in each compact region of $\Rr^2 \setminus (\mathcal{A} \cap \Rr^2)$.
But if the arrangement has symmetry then two critical values can be equal.
\end{question*}

%%%%%%%%%%%%%%%%%%%%%%%%%%%%%%%%%%%%%%%%%%%%%%%%%%%%%%%%%%%%%%%%%%%%%%%
\section{Hyperplane arrangements}

We state a generalization in higher dimension of the equivalence of generic line arrangements.
We have to strengthen the hypothesis: we start with a continuous family of 
\emph{equivalent} hyperplane arrangements (instead of constructing this family).
On the other hand the conclusion is valid for all hyperplane arrangements (that
are Morse outside the arrangement); we do not assume here that hyperplanes are in generic position.

\begin{theorem}
\label{th:hyper}
Let the dimension be $n\not=3$.
Let $(f_t)_{t\in[0,1]}$ be a smooth family of polynomials of
hyperplane arrangements that are Morse functions outside the arrangement.  
If all the arrangements $(f_t=0)$ are equivalent
then the polynomial $f_0$ is topologically equivalent to $f_1$.
\end{theorem}

As seen in the introduction the Morse condition is necessary.
For the proof we can not directly apply the $\mu$-constant theorem
because singularities above $0$ are non-isolated. However the
following lemmas prove that nothing happens at infinity. 

Let us introduce some notations.
Let $B_R$ be the closed $2n$-ball of radius $R$ centred at $(0,\ldots,0)$ in $\Cc^n$
and let $S_R = \partial B_R$.
Let $D_r$ be the closed disk of radius $r$ centred at $0\in\Cc$.
For the polynomial map $f: \Cc^n \to \Cc$, let $T_r = f^{-1}(D_r)$ 
(resp. $T_r^* = f^{-1}(D_r\setminus\{ 0\})$)
be the tube (resp. punctured tube) around the fibre $f^{-1}(0)$.

\begin{lemma}
\label{lem:transv}
Let $f$ be the polynomial of a hyperplane arrangement.
For all $r>0$ there exists a sufficiently large $R$, 
such that for $x \in T^*_r \setminus B_R$, 
$f^{-1}(f(x))$ is transversal to $S_{x/\|x\|}$ at $x$.
\end{lemma}

Then classical arguments of transversality enable the construction of vector fields,
 whose integration gives the following trivialization diffeomorphism $\Phi$:
\begin{lemma}
\label{lem:nosinginfbis}
There exist a diffeomorphism $\Phi$ and $R \gg 1$ such that the
diagram is commutative:
$$ \xymatrix{
{}\mathring B_R \ar[d]_-{f} \ar[r]^-\Phi  & T_r\ar[d]^-{f} \\
  D_r \ar[r]_-{id}  &  D_r     .
}
$$
\end{lemma}

\begin{proof}[Proof of lemma \ref{lem:transv}]
Let $f = \prod_{j=1}^d {\ell_j}$. 
We suppose that the vector space generated by $\{ \grad \ell_j, j=1,\ldots,d \}$ is $n$-dimensional
(if not then we first diminish the dimension of the ambient
space). 

 Suppose that there exists a sequence $(z_k)$ of
points in $\Cc^n$ where the fibres are non transversal to the spheres; that is to say such that:  
$\|z_k\| \rightarrow +\infty$,  $z_k\in T^*_r$, $\grad f(z_k) = \lambda_k z_k$, with $\lambda_k\in
\Cc$.

The hypothesis of the beginning of the proof can be reformulated
as follows: there exists $\ell_j$ such that $ | \ell_j(z_k) |
\rightarrow +\infty$. We denote by $\ell_1, \ldots, \ell_q$ the
linear forms such that $ \ell_j(z_k) \rightarrow 0$. We set $g =
\prod_{j=1}^q{\ell_j}$ and $h = \prod_{j>q}{\ell_j}$. We have
$\deg h > 0$. And as $|f(z_k)|$ is bounded by $r$, we have $q=\deg
g >0$. We can assume that $\bigcap_{j=1,\ldots,q}(\ell_j=0)$ is
given by $(x_1=0,\ldots,x_p=0)$. Then $\ell_j$, $j=1,\ldots,q$ are
linear forms in $(x_1,\ldots,x_p)$ with a constant term equal to
$0$; that is to say $g$ is a homogeneous polynomial in
$\Cc[x_1,\ldots,x_p]$. We write $z_k = (z^k_1,\ldots,z^k_n)$. Hence the $p$ first 
terms $z^k_1,\ldots, z^k_p$ tend towards $0$ as $k \rightarrow
+\infty$. And there exists $i_1>p$ such that $|z_{i_1}^k| \rightarrow
+\infty$.

\bigskip

We now compute $\grad f$. We have $f=g \cdot h$ so that $\frac{\grad f}{f} =
\frac{\grad g}{g} + \frac{\grad h}{h}$. 

But $\frac{\partial h}{\partial x_i}/h =
\sum_{j>q} \frac{\partial \ell_j}{\partial x_i}/\ell_j$ where
$\frac{\partial \ell_j}{\partial x_i}$ is a constant ; $\ell_j(z_k)$
is bounded away from $0$ because $j>q$. Whence $(\frac{\partial
h}{\partial x_i}/h)(z_k)$ and $(\grad h/h)(z_k)$ are bounded as $k
\rightarrow +\infty$.

Moreover if $i>p$ then $\frac{\partial g}{\partial x_i}=0$, it
implies that for $i>p$, $(\frac{\partial f}{\partial x_i}/f)(z_k)$
is bounded as $k\rightarrow +\infty$.

Partial conclusion: there exists $i_1>p$ such that $|z_{i_1}^k| \rightarrow +\infty$ 
and $(\frac{\partial f}{\partial x_{i_1}}/f)(z_k)$ is bounded (as $k\rightarrow +\infty$).

\bigskip

 We now consider the
projection $\pi : \Cc^n \longrightarrow \Cc^p$ defined by
$\pi(x_1,\ldots,x_n)=(x_1,\ldots,x_p)$. Then
$$\langle \pi(\grad f/f) | \pi(z_k) \rangle =\langle \pi(\grad g/g) | \pi(z_k)\rangle +
\langle\pi(\grad h/h) | \pi(z_k)\rangle.$$ 

As $k\rightarrow +\infty$ we have
$\pi(z_k) \rightarrow (0,\ldots,0)$. Then $\langle\pi(\grad h/h) | \pi(z_k) \rangle
\rightarrow 0$. By Euler's relation for the homogeneous polynomial
$g$ of degree $q$  we have $\langle\pi(\grad g/g) | \pi(z_k) \rangle= q>0$. It implies that at least
the module of one component of $\pi(\grad f/f)$ tends towards
$+\infty$. We call $i_2 \le p$ the index of this component.

Partial conclusion: there exists $i_2 \le p$ such that $z_{i_2}^k \rightarrow 0$ 
and $(\frac{\partial f}{\partial x_{i_2}}/f)(z_k) \rightarrow +\infty$  (as $k\rightarrow +\infty$).

\bigskip

We have $\grad f(z_k) = \lambda_k z_k$, hence we have the equality, for all $k$:
$$
z_{i_2}^k \frac{\frac{\partial f}{\partial x_{i_1}}}{f}(z_k) = z_{i_1}^k
\frac{\frac{\partial f}{\partial x_{i_2}}}{f}(z_k).
$$

And we know that, as $k\rightarrow +\infty$, we have:
$z_{i_2}^k \rightarrow 0$, $(\frac{\partial f}{\partial x_{i_1}}/f)(z_k)$ is bounded,
$z_{i_1}^k \rightarrow \infty$, $(\frac{\partial f}{\partial x_{i_2}}/f)(z_k)\rightarrow \infty$.

Then, as $k\rightarrow +\infty$, the left-hand side tends towards $0$ while the right-hand side tends towards $\infty$.
It gives the contradiction.
\end{proof}

\begin{proof}[Proof of the theorem]
We decompose the proof in different steps. 
\begin{enumerate}
  \item Lemma \ref{lem:nosinginfbis} proves that even if $\Binf = \{ 0 \}$ we
only have to prove that $f_0$ and $f_1$ are topologically equivalent
when restricted to $T_r \cap B_R$.

  \item Lemma \ref{lem:cardsing} is valid in any dimension so that $\#\B = \#\Baff =
       2-\chi(f^{-1}(0))$, while the sum $\sum_{c\not=0}{\mu_c} =
       \#\Baff-1$.
  \item Around the fibres above $0$. We fix $R\gg 1$ and
       $0<\epsilon\ll 1$.  We define $T_\epsilon = f^{-1}(D_\epsilon)$.
       We denote $\mathcal{T} = \bigcup_{t\in[0,1]} (T_\epsilon \cap B_R) \times
       \{t\}$. The space $\mathcal{T}$ has a natural Whitney
       stratification given by the intersections of the
       hyperplanes. The function $F : \mathcal{T}
       \longrightarrow D_\epsilon \times [0,1]$ defined by $F(x,t) = (f_t(x),t)$ is a
       function transversal to $\bigcup_{t\in[0,1]} B_R \times
       \{t\}$ (see the proof of lemma \ref{lem:transv}) that satisfies Thom $a_F$ condition.
       From Thom-Mather's second isotopy lemma we have that
       $F$ is a trivial fibration in $t$.
  \item Outside a neighbourhood of the fibres $f_t^{-1}(0)$. We can
       apply the methods of the proof of the global $\mu$-constant
       theorem (theorem \ref{th:mucst}). It provides a trivialization, which can be glued with
       the one obtained around the fibre above $0$. It proves that
       the polynomials $f_0$ and $f_1$ are topologically
       equivalent on $T_r \cap B_R$. As nothing happens at infinity
       the polynomials $f_0$ and $f_1$ are topologically
       equivalent on $\Cc^n$.
\end{enumerate}
\end{proof}

%%%%%%%%%%%%%%%%%%%%%%%%%%%%%%%%%%%%%%%%%%%%%%%%%%%%%%%%%%%%%%%%%%%%%%%
\section{Nearby fibre and intermediate links}

We end with a topological description of the generic fibre of a line arrangement.
In fact we will firstly prove a result for all polynomials and next will apply this to
compute the topology of generic fibres of a line arrangement intersected with any a ball.

%%%%%%%%%%%%%%%%%%%%%%%%%%%%%%%%%%%%%%%%%%%%%%%%%%%%%%%%%%%%%%%%%%%%%%%
\subsection{Nearby fibre}

It is useful to consider a smooth deformation $(f=\delta)$ of $(f=0)$.
But we shall also consider a deformation $f_s$ of the polynomial $f$. We
will prove that the generic fibres $(f=\delta)$ and $(f_s=\delta)$, restricted to a ball, are diffeomorphic.

\begin{lemma}
\label{lem:trans}
Let $f \in \Cc[x_1,\ldots,x_n]$. Suppose that $(f=0)$ has isolated (affine)
singularities. For any $r>0$ such that the intersection $(f=0)$ with the sphere
$S^{2n-1}_r(0)$ is transversal, there exists $0 < \delta  \ll 1$ such that:
\begin{enumerate}
  \item The intersection of $(f=c)$ with $S^{2n-1}_r(0)$ is transversal, for any $c\in D^2_\delta(0)$.
  \item The restriction $f : f^{-1}(D^2_\delta(0)\setminus\{ 0 \}) \cap B^{2n}_r(0) \to D^2_\delta(0) \setminus\{ 0 \}$
is a locally trivial fibration.
\end{enumerate}
\end{lemma}

\begin{proof}
 Follows from the continuity of the transversality for the first item.
The second one follows from Ehresmann fibration theorem.
\end{proof}

It is crucial that we first choose the radius of the sphere and then the radius of the disk.
This is the opposite of the situation for singularity at infinity.

%%%%%%%%%%%%%%%%%%%%%%%%%%%%%%%%%%%%%%%%%%%%%%%%%%%%%%%%%%%%%%%%%%%%%%%
\subsection{Deformation of $f$}

It is useful not to consider $(f=0)$ but to study a deformation 
$(f_s=0)$ of $f$. The following proposition proves that the nearby fibre
$(f=\delta)$ and $(f_s=\delta)$ are diffeomorphic.

\begin{proposition}
\label{prop:nearby}
Let $f$, $r>0$, $0 < \delta \ll 1$ as in Lemma \ref{lem:trans}.
Consider a deformation $f_s$, $s \in [0,1]$ with $f_0 = f$.
For all sufficiently small $0 < s \ll \delta$ we have :
\begin{enumerate}
  \item The intersection of $(f_s=c)$ with $S^{2n-1}_r(0)$ is transversal, for any $c\in D^2_\delta(0)$.

  \item The sum of Milnor numbers of the critical points of $f_s$ inside 
$f^{-1}(D^2_\delta(0)) \cap B_r^{2n}(0)$ equals the sum of Milnor numbers of the critical points of $(f=0)$ inside 
$B^{2n}_r(0)$.

  \item  The restriction $f_s : f_s^{-1}(D^2_\delta(0)\setminus \mathcal{B}_s) \cap B^{2n}_r(0) 
\to D^2_\delta(0) \setminus \mathcal{B}_s$
is a locally trivial fibration, where $\mathcal{B}_s$ is a finite number of points 
(the critical values of $f_s$ corresponding to 
the critical points of $f_s$ inside $f^{-1}(D^2_\delta(0)) \cap B^{2n}_r(0)$).

  \item The fibrations  $f : f^{-1}(S^1_\delta(0)) \cap B^{2n}_r(0) \to S^1_\delta(0)$ and
$f_s : f_s^{-1}(S^1_\delta(0)) \cap B^{2n}_r(0) \to S^1_\delta(0)$ are diffeomorphic. In particular
the fibres $(f=\delta)$ and $(f_s=\delta)$ are diffeomorphic.
\end{enumerate}
\end{proposition}

\begin{proof}
Very standard proofs: continuity of the critical points, continuity of transversality, Ehresmann fibration theorem,
integration of vector fields.
\end{proof}

Be careful! 
The order for choosing the constants is crucial.
We fix $f=f_0$, then $r$, then $\delta$, then $s$.
The $\delta$ is chosen small for $f_0$ but is not small for $f_s$.
In other words the fibration $f_s : f_s^{-1}(S^1_\delta(0)) \cap B^{2n-1}_r(0) \to S^1_\delta(0)$
is \textbf{not} diffeomorphic to the fibration 
$f_s : f_s^{-1}(S^1_{\delta'}(0)) \cap B^{2n-1}_r(0) \to S^1_{\delta'}(0)$
for all $0 < \delta' \ll 1$. Even if the fibres are diffeomorphic!

\begin{remark*}
It should also be compared to the work of Neumann and Rudolph, \cite{NR}.
Another interesting idea is to include the smooth part of $(f=0)$ into the nearby fibre
$(f=\delta)$, see \cite{Bod1}.  
\end{remark*}

%%%%%%%%%%%%%%%%%%%%%%%%%%%%%%%%%%%%%%%%%%%%%%%%%%%%%%%%%%%%%%%%%%%%%%%
%%%%%%%%%%%%%%%%%%%%%%%%%%%%%%%%%%%%%%%%%%%%%%%%%%%%%%%%%%%%%%%%%%%%%%%
\section{Nearby fibres of line arrangements}

The case where $(f=0)$ is an affine line arrangement is of particular interest.
Let $f(x,y) = \prod_i(a_i x + b_i y + c_i)$, $a_i, b_i, c_i \in \Cc$.
Growing the radius of the sphere yields only two phenomena: birth or joint
(see \cite{Bor}).
There is no point of non-transversality for $(f=0)$.
We can compute the topology of a nearby fibre $(f=\delta)\cap B_r$,
where $B_r = B^4_r(0)$.

\begin{proposition}
\label{prop:genfiber}
The topology of $(f=\delta) \cap B_r$ can be computed as follows:
for each line $\ell_i$ of $(f=0)$ that intersects $B_r$, associate a $2$-disk $D_i$.
For any two lines $\ell_i$ and $\ell_j$ that intersect in $B_r$ glue the disks $D_i$ to $D_j$
by two twisted-bands $[0,1]\times[0,1]$.
\end{proposition}

A fundamental example is an arrangement of only two lines that intersect inside
$B_r$. Then $(f=\delta)$ is two disks joined by two twisted bands.
See the pictures from left two right. (i) a real picture of an arrangement with two lines (in red) that intersect
in the ball (in blue), and a real picture of one generic fibre.
(ii) and (iii) a topological picture of the generic fibre as the gluing of two disks by two bands.

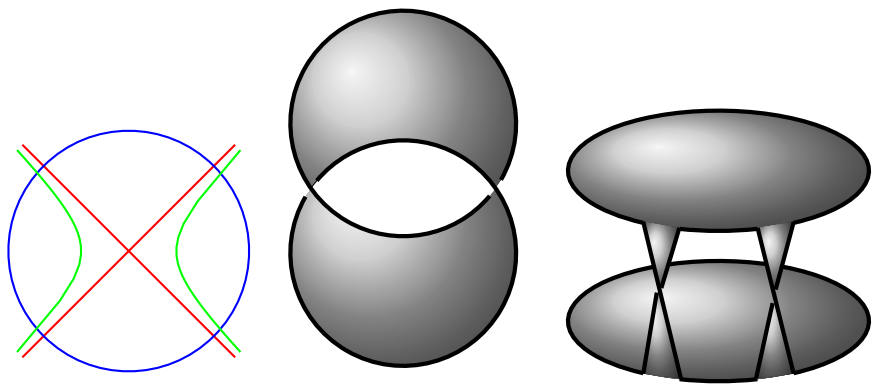
\begin{figure}[H]
\begin{tikzpicture}[scale=2]

\begin{scope}[rotate=-45]

\draw[red, thick] (0,-1)--(0,1);
\draw[red, thick] (-1,0)--(1,0);

\draw[blue, thick] (0,0) circle (0.8cm);

\draw [thick, color=green, domain=0.05:1] plot(\x,0.05*1/\x);
\draw [thick, color=green, domain=0.05:1] plot(-\x,-0.05*1/\x);
\end{scope}
\end{tikzpicture}
%%%
\begin{tikzpicture}[scale=1.5]

\fill [ball color=gray!50] (-30:1) arc (-30:320:1);

\begin{scope}[yshift=-1.15cm]
  \fill [ball color=gray!50] (-210:1) arc (-210:140:1);
\end{scope}

 \begin{scope}
   \clip  (0,0) circle (1cm);
   \draw [fill=white]   (0,-1.15) circle (1cm);
\end{scope}

\draw[ultra thick] (-30:1) arc (-30:320:1);

\begin{scope}[yshift=-1.15cm]
  \draw[ultra thick] (-210:1) arc (-210:140:1);
\end{scope}
\end{tikzpicture}
%%%
\begin{tikzpicture}[scale=2]

\begin{scope}[yshift=0.5cm,yscale=0.4]
  \fill [ball color=gray!50] (0,0) circle (1cm);
   \draw[ultra thick] (-60:1) arc (-60:240:1);
   \draw[ultra thick] (-105:1) arc (-105:-75:1);
\end{scope}

\begin{scope}[yshift=-0.5cm,yscale=0.4]
  \fill [ball color=gray!50] (0,0) circle (1cm);
   \draw[ultra thick] (-60:1) arc (-60:240:1);
   \draw[ultra thick] (-105:1) arc (-105:-75:1);
\end{scope}

% Left band

\coordinate (A) at (-0.5,0.16);
\coordinate (B) at(-0.25,-0.88);
\coordinate (C) at (-0.26,0.12);
\coordinate (D) at(-0.5,-0.85);

\coordinate (E) at (-0.39,-0.29);
\coordinate (EE) at (-0.38,-0.28);
\coordinate (F) at (-0.40,-0.30);
\coordinate (FF) at (-0.41,-0.31);

\fill[ball color=gray!50] (A)--(E)--(C)--cycle;
\fill[ball color=gray!50] (D)--(F)--(B)--cycle;

\draw[ultra thick] (A)--(B);
\draw[ultra thick] (C)--(EE);
\draw[ultra thick] (FF)--(D);

%%%% Right band

\coordinate (nA) at (0.5,0.16);
\coordinate (nB) at(0.25,-0.88);
\coordinate (nC) at (0.26,0.12);
\coordinate (nD) at(0.5,-0.85);

\coordinate (nE) at (0.37,-0.28);
\coordinate (nEE) at (0.38,-0.29);
\coordinate (nF) at (0.38,-0.30);
\coordinate (nFF) at (0.36,-0.38);

\fill[ball color=gray!50] (nA)--(nE)--(nC)--cycle;
\fill[ball color=gray!50] (nD)--(nF)--(nB)--cycle;

\draw[ultra thick] (nC)--(nD);
\draw[ultra thick] (nA)--(nEE);
\draw[ultra thick] (nFF)--(nB);

\end{tikzpicture}
\end{figure}

\begin{proof}
First of all we make a deformation of the arrangement $f$ to another arrangement $f_s$,
such that $(f_s=0)$ is a union of lines with only double points.
Of course, some other singular fibres appear, but by proposition \ref{prop:nearby},
the nearby fibres $(f=\delta)$ and $(f_s=\delta)$ are diffeomorphic.

Then we have to consider only gluing of nearby fibres of intersection of two lines as in the example 
above.
\end{proof}

Illustration. To compute the topology of the generic fibre of this arrangement with $4$ lines (left),
we first shift one line in order to remove the triple point (right picture). We also have drawn the real trace 
of a generic fibre (in red). The topology is built as follows starts from $4$ disks labelled $D_1,\ldots, D_4$ 
(corresponding two the lines $\ell_1,\ldots,\ell_4$).
Between each pair $(D_1,D_2)$, $(D_1,D_3)$, $(D_2,D_3)$ $(D_1,D_4)$ attach two bands 
(each two bands corresponds to one intersection inside the ball).
The Euler characteristic of a generic fibre is then $-4$ and as a surface it is a torus with $4$ punctures.
\begin{figure}[H]
\begin{tikzpicture}[scale=3]
\begin{scope}[rotate=30]
\draw[red, thick] (0,-1)--(0,1);
\draw[red, thick] (-1.5,0)--(0.5,0);
\draw[red, thick] (-0.9,-0.9)--(0.4,0.4);
\draw[red, thick] (-1.5,-0.6)--(-0.2,1);
\draw[blue, thick] (-0.5,0) circle (0.8cm);
\end{scope}
\end{tikzpicture}  
\begin{tikzpicture}[scale=3]

\begin{scope}[rotate=30]
\draw[red, thick] (0,-1)--(0,1);
\draw[red, thick] (-1.5,0)--(0.5,0);
\draw[red, thick] (-1,-0.9)--(0.4,0.6);
\draw[red, thick] (-1.5,-0.6)--(-0.2,1);
\draw[blue, thick] (-0.5,0) circle (0.8cm);

\draw [green, thick] plot [smooth] coordinates {(0.4,0.05) (0.05,0.05) (0.05,0.18) (0.3,0.4)};
\draw [green, thick] plot [smooth] coordinates {(-0.9,-0.9) (-0.2,-0.1) (-0.04,-0.1) (-0.05,-0.9)};
\draw [green, thick] plot [smooth] coordinates {(-1.5,-0.05) (-1.1,-0.05) (-1.5,-0.5)};
\draw [green, thick] plot [smooth] coordinates {(-0.05,1)(-0.05,0.3) (-0.2,0.08) (-0.8,0.1)(-0.15,1)};
       \node at (0.6,0) [color=red] {$\ell_1$};
       \node at (0.5,0.65) [color=red] {$\ell_2$};
       \node at (0,1.1) [color=red] {$\ell_3$};
       \node at (-0.2,1.1) [color=red] {$\ell_4$};
\end{scope}
\end{tikzpicture}
\end{figure}

%%%%%%%%%%%%%%%%%%%%%%%%%%%%%%%%%%%%%%%%%%%%%%%%%%%%%%%%%%%%%%%%%%%%%%%

\bigskip
\bigskip

\end{document}